\documentclass{cmslatex}
\usepackage{latexsym, amssymb, enumerate, amsmath}
\usepackage{url}
\usepackage{graphics,graphicx,subfig,color}
\usepackage{bm}
\usepackage[numbers]{natbib}
\renewcommand{\theequation}{\arabic{section}.\arabic{equation}}

\newtheorem{thm}{thm}[section]

\newtheorem{remark}[thm]{Remark}
\newcommand{\eps}{{\varepsilon}}

\newcommand{\D}{\mathrm{D}}

\renewcommand*{\Xi}{{\boldsymbol{\xi}}}

\newcommand*{\I}{{\mathbb{I}}}

\newcommand*{\R}{{\mathbb{R}}}

\newcommand{\LE}{{\mathcal{L}}_G}

\newcommand{\phieta}{ \varphi_\eta}
\newcommand{\phiyn}{ \varphi_{y^n}}
\newcommand{\phiYn}{ \varphi_{Y{(t^n)}}}

\newcommand{\dt}[0]{{\delta t}}
\newcommand{\Dt}[0]{{\Delta t}}
\newcommand{\td}[0]{t_\Delta}

\newcommand{\twoparts}[4]
{
	\left\{
		\begin{array}{lll}
			&\displaystyle#1 &\quad \textnormal{if } \displaystyle0<\dt\le\frac{2\eps}{\lambda+1} \vspace{6pt}\\ 
			&{\displaystyle#2 }& {\quad \textnormal{if } \displaystyle\frac{2\eps}{\lambda+1}<\dt<\frac{2\eps}{\lambda}}
		\end{array}
	\right.
}

\allowdisplaybreaks

\begin{document}

\title{On convergence of the projective integration method for stiff ordinary differential equations}
%\thanks{%{Received date / Revised version date}
          % The correct dates will be entered by the CMS editor}}

          %For each author, make a block with the following four macros:
\author{John Maclean
\thanks {School of Mathematics and Statistics, University of Sydney, NSW 2006. Australia. Email: j.maclean@maths.usyd.edu.au.}
\and Georg A. Gottwald \thanks {School of Mathematics and Statistics, University of Sydney, NSW 2006. Australia. Email: georg.gottwald@sydney.edu.au.}}
          %{Put the URL for your home page here if you have one}

          %Use \thanks statements for acknowledgements of grants and
          %support. They will appear below all the authors' addresses, so be
          %specific about which author is thanking whom:

          %\thanks{}

\pagestyle{myheadings} \markboth{On convergence of the projective integration method}{John Maclean and Georg A. Gottwald}

\maketitle

\begin{abstract}
We present a convergence proof of the projective integration method for a class of deterministic  multi-dimensional multi-scale systems which are amenable to centre manifold theory. The error is shown to contain contributions associated with the numerical accuracy of the microsolver, the numerical accuracy of the macrosolver and the distance from the centre manifold caused by the combined effect of micro- and macrosolvers, respectively. We corroborate our results by numerical simulations. 
\end{abstract}


%%%%%%%%%%%%%%%%%%%%%%%%%%%%%%%%%%%%%%%

\section{Introduction}

Devising efficient computational methods to simulate high-dimensional complex systems is of paramount importance to a wide range of scientific fields, ranging from nanotechnology, biomolecular dynamics and material science to climate science. The dynamics of large complex systems is made complicated by their high dimensionality and by the possible existence of active entangled processes running on temporal scales varying by several orders of magnitude. To resolve all variables in a high-dimensional system and capture the whole range of temporal scales is impossible given current computing power. In many applications, however, the modeller is only interested in the dynamics of a few relevant slow macroscopic coarse-grained variables. How to extract from a dynamical system the relevant dynamics of the slow degrees of freedom while ensuring that the collective effect of the unresolved variables is implicitly represented is one of the most challenging problems in computational modelling. There are two separate scenarios when such a dimension reduction is possible: scale separation and weak coupling \cite{Givonetal04}. We restrict ourselves here to the big class of time scale separated systems, in particular to deterministic stiff dissipative systems.\\
%Despite the obvious computational advantage of reducing the dimensionality, reductions of scale-separated systems offer the additional advantage that the time-step to be used in simulations can be orders of magnitudes larger since the reduced model only involves slow variables.\\

\noindent
%On the theoretical side centre manifold theory can be employed for a large class of dissipative dynamical systems with time-scale separation \cite{Carr} to derive effective reduced equations for the slow variables. 
Recently two numerical methods to deal with multi-scale systems have received much attention, the \emph{projective integration method} (PI) within the equation-free framework \cite{GearKevrekidis03,KevrekidisGearEtAl03} and the \emph{heterogeneous multiscale method} (HMM) \cite{EEngquist03,EngquistTsai05}. These powerful methods have successfully been applied to a wide range of problems, including modelling of water in nanotubes, micelle formation, chemical kinetics and climate modelling \cite{Majdaetal01,EEtAl05,KevrekidisSamaey09}. The general strategy of both methods is the following: Provided %the relevant slow variables are known and 
there exist closed (but possibly unknown) equations for the slow variables, the simulation is split between a macrosolver 
%for the slow variables 
employing a large integration time-step and a microsolver, in which the full high dimensional system is integrated for a short burst employing a small integration time step. The two methods differ in the way the information of the microsolver is utilized in the macrosolver to evolve the slow variables. In the equation-free approach this is done without any assumptions on the actual form of the reduced equations, using the microsolver to estimate the temporal derivative of the slow variable, which is then subsequently propagated over a large time step. The underlying assumption is that the fast variables quickly relax to their equilibrium value (conditioned on the slow variables), and that subsequently the dynamics evolves on a reduced slow manifold. In heterogeneous multiscale methods, on the other hand, information available (i.e. through perturbation theory such as for example centre manifold theory, averaging theorems or homogenisation \cite{Carr,Khasminsky66,Kurtz73,PavliotisStuart}) is used to determine the functional form of the reduced slow vector field; the microsolver is then used to estimate the coefficients of this slow vector field conditioned on the slow variables. For further details we refer the reader to \cite{KevrekidisSamaey09,KevrekidisGearEtAl03,EEtAl07,VandenEijnden07}.\\
%In the equation-free approach this is done without any assumptions on the actual form of the reduced equations, using the output of the slow variables after the microstep. In the heterogeneous multiscale methods the vectorfield of the slow reduced equation is estimated employing averaging theorems for singularly perturbed systems \cite{Carr,Khasminsky66}.\\
%The heterogeneous multiscale methods do so effectively by approximating averages over the invariant measure induced by the fast variables using temporal averages and short integrations of the fast dynamics.

\noindent
The question we address here is the convergence of the numerical PI approximation to the true solution of a full multi-scale system of stiff deterministic ordinary differential equations. There exists a body of analytical work on the convergence of HMM \cite{EEngquist03,E03,EngquistTsai05}, including formulations for stochastic differential equations \cite{ELiuVandenEijnden05,Liu10}. Convergence results for PI were obtained for specific deterministic systems and for on-the-fly local error estimates in numerical algorithms \cite{GearKevrekidis03,GearLee06}. Stochastic differential equations were treated in \cite{GivonEtAl06}. Here we provide global error estimates for PI for a subclass of systems amenable to centre-manifold theory.\\

\noindent
The paper is organized as follows. In Section~\ref{sec:model} we discuss the class of dynamical systems studied, and briefly summarize in Section~\ref{sec:methods} HMM and PI for these systems. In Section~\ref{sec:proof}, the main part of this work, we derive rigorous error bounds for those numerical multiscale methods. In Section~\ref{sec:numerics} these are numerically verified. We conclude with a summary in Section~\ref{sec:summary}.

%%%%%%%%%%%%%%%%%%%%%%%%%%%%%%%%%%%%%%%

\section{Model}
\label{sec:model}

We consider deterministic multiscale systems for slow variables $y_\eps\in\R^n$ and fast variables $x_\eps\in\R^m$ of the form
\begin{align}
\dot y_\eps &=g(x_\eps,y_\eps)\; ,
\label{baseslow}
\\
\dot x_\eps &= \frac{1}{\varepsilon}(- \Lambda \, x_\eps+ f(y_\eps))\; ,
\label{basefast}
\end{align}
with time scale separation parameter $0<\varepsilon \ll 1$. Without loss of generality we assume that $(x_\eps,y_\eps)=(0,0)$ is a fixed point.  We assume there is a coordinate system such that the matrix $\Lambda\in \R^{m\times m}$ is diagonal with diagonal entries $\lambda_{ii}>0$. We further assume that we allowed for a scaling of time such that $\min(\lambda_{ii}) =1 $ and define $\max(\lambda_{ii})=\lambda$. We assume that the vectorfield of the slow variable $g(x,y)$ is purely nonlinear with $\det \D g(0,0)\neq 0$, where $\D g$ denotes the Jacobian of $g(x,y)$, so that there exists a centre manifold $x=h_\eps(y)$ on which the slow dynamics evolves according to
\begin{align}
\dot Y = G(Y)\; ,
\label{e.CMT}
\end{align}
with $Y=y+{\mathcal{O}}(\eps)$. The reduced slow vectorfield is given by (see for example \cite{Carr})
\begin{align}
\label{e.G0}
G(y) &= g(h_\eps(y),y)\; .
\end{align}
Initial conditions close to the fixed point are exponentially quickly attracted towards the centre manifold along the stable manifold of the fast variables $x_\eps$. Near the centre manifold the dynamics slows down and is approximately determined by the dynamics of the slow variables (\ref{e.CMT}) only. Centre manifold theory assures that on times $T\sim{\mathcal{O}}(1)$ the dynamics of (\ref{baseslow}) is well approximated by the reduced slow system (\ref{e.CMT}) (see for example \cite{Carr}).\\

It is worthwhile to formulate centre manifold theory in the framework of averaging \cite{Givonetal04}, and view the effect of the fast variables on the slow variables through their induced empirical measure which is approximated by $\mu_y(dx) = \delta(x-h_\eps(y))dx$, conditioned on the slow variables. The dynamical system (\ref{baseslow})-(\ref{basefast}) can equivalently be described by its associated Liouville equation for the probability density function $\rho(x_\eps,y_\eps)$. After an initial transient, the solution of the Liouville equation is approximated by
\begin{align}
\rho(x,y) = \delta(x-h_\eps(y))\hat\rho(y) + \mathcal{O}(\varepsilon)\; .
\end{align}
We can now reformulate (\ref{e.G0}) as
\begin{align}
\label{e.G}
G(y) &= \int g(x,y) \mu_y(dx)\nonumber\\
&= g(h_\eps(y),y)
\; .
\end{align}
The underlying assumption is that the measure $\mu_y(dx) = \delta(x-h_\eps(y))dx$ is the physically observable measure on the $y$-fibre;  that is, Lebesgue almost all initial conditions of the fast variables will evolve to generate $\mu_y(dx)$, if sufficiently close to the centre manifold. Centre manifold theory \cite{Carr} makes this approach rigorous and $\hat \rho(y)$ is the invariant density of the reduced system (\ref{e.CMT}). This is exploited explicitly in HMM and implicitly in PI. In the following we will approximate the centre manifold $h_\eps(y)$ by $h_\eps(y)\approx  \Lambda^{-1} f(y)+{\mathcal{O}}(\eps)$.

%%%%%%%%%%%%%%%%%%%%%%%%%%%%%%%%%%%%%%%

\section{Numerical Multiscale Methods}
\label{sec:methods}

We consider here two numerical methods to deal with the deterministic multi-scale system (\ref{baseslow})-(\ref{basefast}), namely heterogeneous multiscale methods and projective integration. 
We use the formulation of PI as in \cite{GearKevrekidis03,KevrekidisGearEtAl03} and of heterogeneous multiscale methods as in \cite{EEngquist03,EngquistTsai05}. We will see that both methods can be formulated in the same framework.\\

\noindent
We denote by $x^n$ and $y^n$ the numerical approximations to the solutions of (\ref{baseslow})-(\ref{basefast}) evaluated at the discrete time $t^n$, $x_\eps(t^n)$ and $y_\eps(t^n)$ respectively. Further we denote by $Y(t^n)$ the time-continuous solution of the reduced ordinary differential equation (\ref{e.CMT}) evaluated at time $t^n$. Throughout the paper superscripts denote discrete variables whereas brackets are reserved for continuous variables. We will employ a forward Euler method for both the micro- as well as the macrosolver.\\

%%%%%%%%%%%%%%%%%%%%%%%%%%%%%%%%%%%%%%%

\subsection{Heterogeneous Multiscale Methods}

The microsolver consists of $M$ microsteps with micro-time step $\dt$ \cite{EEngquist03}. The slow variables are held fixed during the microsteps, assuming infinite time scale separation. Let us denote by $x^{n,m}$ the $m${\rm th} microstep of the fast variables which was initialized at the $n${\rm th} macrostep at $t^n=n\Dt$ with $x^{n,0} = x^{n-1,M}=x^n$ and $y^{n,0} = y^n$. The microstep for a forward Euler scheme then is
\begin{align*}
x^{n,m+1} = x^{n,m} - \frac{\dt}{\varepsilon} \left( \Lambda  x^{n,m}-f(y^n) \right)\; ,
\end{align*}
for $m=0,\dots,M$. 
%\textcolor{blue}{Invoking the Birkhoff ergodic theorem a microsolver is employed to propagate the fast dynamics for fixed $y$ to estimate the averages (\ref{e.G}) via a temporal average over a trajectory. This is justified as long as the slow variable remains constant while integrating over the fast fibres. }
Invoking the Birkhoff ergodic theorem, the time series of the fast variables $x^{n,m}$ is used to estimate the average in the reduced slow vectorfield (\ref{e.G}), and the macrosolver is initialized at the previous macrostep $y^n$ and is written for a forward Euler step as
\begin{align}
y^{n+1} = y^n + \Dt \, \hat{g}(x^n,y^n)
\label{e.HMM0}
\end{align}
with the time-discrete approximation of the reduced slow vectorfield (\ref{e.G0})
\begin{align*}
%\hat{g}(y^n) = \sum_{m=0}^M W_m g(x^{n,m},y^n) 
G(y_\varepsilon(t^n))\approx \hat{g}(x^n,y^n) = \sum_{m=0}^M W_m g(x^{n,m},y^n) 
\end{align*}
for some weights $W_m$. This is justified as long as the slow variable remains constant while integrating over the fast fibres. 
%The approximation of the averaged reduced slow vectorfield (\ref{define_g_twiddle}) is coined F-estimator in \cite{EEngquist03}. 
Being associated with the empirical approximation of the invariant density induced by the fast dynamics (cf. (\ref{e.G})), the weights $W_m$ satisfy the normalization constraint
\begin{align}
\sum_{m=0}^M W_m = 1\; .
\label{e.Wnormal}
\end{align}
The weights should be chosen to best approximate the invariant density of the fast variable. A natural choice for our system (\ref{baseslow})-(\ref{basefast}), where the fast dynamics rapidly settles on the centre manifold, and does not advance significantly on the centre manifold in $M$ microsteps is to set
\begin{align} 
W_m = \left\{ \begin{gathered}
0 \;\;\;\text{ for } m<M \\
1\;\;\;\text{ for } m=M
\end{gathered} \right.
\; .
\label{W_HMM0}
\end{align}
The convergence of this scheme has been analyzed in \cite{E03}. More general weights were discussed in \cite{EngquistTsai05}.\\ 

\noindent
It is pertinent to mention that for this formulation of HMM the fast variables have to be initialized before each application of the microsolver; see for example \cite{EEtAl07} for a discussion on initializations. In PI, as described in the next subsection, the fast variables are initialized on the fly.\\ 
%As we will see in Section~\ref{sec:numerics} this may be, depending on the functional form of the approximate centre manifold $f(y)$, advantageous or disadvantageous.\\
%\note{As we will see in Section~\ref{sec:numerics} this may be, depending on the functional form of the approximate centre manifold $f(y)$, advantageous or disadvantageous.: Do we really show this?}

\noindent
The {\emph{seamless}} heterogeneous multiscale method \cite{E03,EngquistTsai05} can be applied to the case when the slow and fast variables are not explicitly known. The seamless HMM formulation for our system (\ref{baseslow})-(\ref{basefast}) advances both, fast and slow, variables simultaneously first through $M$ microsteps, and then subsequently through one macrostep. We adopt here the formulation described in \cite{E03} for deterministic systems. Introducing $y^{n,m}$ analogously to $x^{n,m}$ the $M$ microsteps are executed as
 \begin{align}
\label{micro_x} x^{n,m+1} &= x^{n,m} - \frac{ \dt}{ \varepsilon}( \Lambda \, x^{n,m} - f(y^{n,m}) ) \\
\label{micro_y} y^{n,m+1} &= y^{n,m} + \dt \, g(x^{n,m},y^{n,m}) \; . 
\end{align}
The macrostep is initialized at $(x^n,y^n)=(x^{n,0},y^{n,0})$ and takes the form
\begin{align}
\label{Smacro_x}
x^{n+1} &= x^n - \frac{\Dt}{\varepsilon} \sum_{m=0}^M W_m \left( \Lambda \, x^{n,m} - f(y^{n,m}) \right) \; , \\
y^{n+1} &= y^n + \Dt \, \tilde{g}(x^n,y^n)
\; ,
\end{align}
with the time-discrete approximation of the slow vectorfield 
\begin{align}
\label{define_g_twiddle}
\tilde{g}(x^n,y^n) := \sum_{m=0}^M W_m g(x^{n,m},y^{n,m}) \; .
\end{align}
Again, the weights $W_m$ need to satisfy (\ref{e.Wnormal}) (see \cite{EngquistTsai05} for choices of weights $W_m$).

%%%%%%%%%%%%%%%%%%%%%%%%%%%%%%%%%%%%%%%

\subsection{Projective Integration}
We present a formulation of projective integration which is close to the seamless formulation of HMM. 
%Projective Integration is a form of equation-free solver, a term used to stress that these schemes do not assume knowledge of the explicit functional form of the reduced slow vectorfield (\ref{e.G}). 
PI advances both slow and fast variables at the same time, as done in seamless HMM. However, the PI scheme proposed in \cite{GearKevrekidis03} does not start the macrostep at $(x^{n,0},y^{n,0})$ as done in our adopted formulation of seamless HMM, but at $(x^{n,M},y^{n,M})$, allowing for nontrivial dynamics of the slow variables during the $M$ microsteps. This takes into account finite time scale separation $\varepsilon$, ignored in HMM. The microsteps of PI are exactly as in seamless HMM, (\ref{micro_x})--(\ref{micro_y}). The macrostep of PI is given by estimating the temporal time derivative of the slow (and fast) variables using the microsolver according to 
\begin{align*}
x^{n+1} &= x^{n,M} + \Dt \, \frac{x^{n,M} - x^{n,M-1}}{\dt}\\
y^{n+1} &= y^{n,M} + \Dt \, \frac{y^{n,M} - y^{n,M-1}}{\dt}
\; ,
\end{align*}
which becomes upon using the Euler step of the microsolver (\ref{micro_x})--(\ref{micro_y})
\begin{align}
\label{PImacro_x} x^{n+1} &= x^{n,M} - \frac{ \Dt}{ \varepsilon}( \Lambda \, x^{n,M} - f(y^{n,M}) )\\
\label{PImacro_y} y^{n+1} &= y^{n,M} + \Dt \, g(x^{n,M},y^{n,M})
\; .
\end{align}
Here the time between two macrosteps is $\td=\Dt+M\dt$ and we have $t^n = n\td$.

\noindent
Upon substituting the microsteps (\ref{micro_x})--(\ref{micro_y}) into the macrosolver (\ref{PImacro_x})--(\ref{PImacro_y}) we reformulate PI, such that it formally resembles seamless HMM, as
 \begin{align}
 x^{n+1} 
%& = x^n  - \frac{\dt}{\varepsilon} \sum_{m=0}^{M-1} \left( \Lambda \, x^{n,m} - f(y^{n,m}) \right) - \frac{ \Dt}{ \varepsilon}( \Lambda \, x^{n,M} - f(y^{n,M}) )
%\nonumber
%\\
&= x^n - \frac{\td}{\varepsilon} \sum_{m=0}^M W_m \left( \Lambda \, x^{n,m} - f(y^{n,m}) \right)\; ,
\label{e.PIHMM}
\\
y^{n+1} 
% &= y^n + \dt \sum_{m=0}^{M-1} g(x^{n,m},y^{n,m}) + \Dt\, g(x^{n,M},y^{n,M})
% \nonumber
% \\
 &= y^n + \td \, \tilde{g}(x^n,y^n) \; ,
 \label{e.PIHMMy}
\end{align}
with the weights $W_m$ now defined as
\begin{align}
W_m =\left\{ \begin{gathered}
\frac{\dt}{\td} \;\;\;\text{ for } m<M
\\
\frac{\Dt}{\td} \;\;\;\text{ for } m=M
\end{gathered} \right. \;.
\label{e.WPI}
\end{align}
Hence PI with a microstep of $\dt$ and macrostep of $\Dt$ is equivalent to seamless HMM with a microstep of $\dt$, a macrostep of $\td$ and a particular choice of weights $W_m$.

%%%%%%%%%%%%%%%%%%%%%%%%%%%%%%%%%%%%%%%

\section{Error analysis for Projective Integration}
\label{sec:proof}

We will now provide rigorous error bounds for PI in the formulation (\ref{e.PIHMM})--(\ref{e.WPI}). We follow the general line of proof used in \cite{E03} for error bounds of HMM with the choice of weights (\ref{W_HMM0}).\\

\noindent
Throughout this work we assume the following conditions on the global growth and smoothness of solutions of our system and on the numerical discretization parameters of PI.\\

%
%\begin{assumptions}

\noindent
{\sc Assumptions}
\vspace{-4mm}
\begin{itemize}
\renewcommand{\labelitemi}{}
\item 
\renewcommand{\labelitemi}{\it A1:}
 \item The vectorfield $f(y)$ is Lipschitz continuous; that is there exists a constant~$L_f$ such that 
\begin{align*}
|f(y_1)-f(y_2)| \leq L_f|y_1 - y_2|\; .
\end{align*}
\renewcommand{\labelitemi}{\it A2:}
\item The vectorfield $g(x,y)$ is Lipschitz continuous; that is there exists a constant~$L_g$ such that 
\begin{align*}
|g(x_1,y_1)-g(x_2,y_2)| \leq L_g(|x_1-x_2|+|y_1-y_2|)\; .
\end{align*}
\renewcommand{\labelitemi}{\it A3:}
\item The vectorfield $f(y)$ is bounded for all $y$; that is there exists a constant~$C_f$ such that 
\begin{align*}
C_f =  \sup|f(y)| \;.
\end{align*}
\renewcommand{\labelitemi}{\it A4:}
\item The vectorfield $g(x,y)$ is bounded for all $x,y$; that is there exists a constant~$C_g$ such that 
\begin{align*}
C_g =  \sup|g(x,y)| \;.
\end{align*}
\renewcommand{\labelitemi}{\it A5:}
\item The reduced vectorfield $G(Y)$ of the slow dynamics is of class $C^1$; that is there exists a constant~$C^*$ such that 
\begin{align*}
C^* = \sup|\ddot{Y}(t)| \;.
\end{align*}
\renewcommand{\labelitemi}{\it A6:}
\item The microstep size $\dt$ resolves the fastest of the fast variables, while the macrostep size~$\Dt$ resolves the dynamics of the slow system, so that
\begin{align*}
0 < \dt \le \frac{2\eps}{\lambda}  < \Dt \; .
\end{align*}
\renewcommand{\labelitemi}{\it A7:}
\item The total time $M\dt$ of the microsteps is sufficiently short so that
\begin{align*}
 L_G M \dt \le L_g (1 + L_f) M \dt   \leq 1\;,
\end{align*}
where $L_G \leq L_g(1+L_f)$ is the Lipschitz constant of the reduced dynamics \eqref{e.CMT}.\\
\renewcommand{\labelitemi}{\it A8:}
\item The macrostep size $\Dt$, number of microsteps $M$ and  microstep size $\dt$ are chosen such that 
\begin{equation*}
                   \begin{array}{lll}
			&\displaystyle\Dt  \, \exp\left(-\frac{M\dt}{\eps}\right) < \frac{\eps}{\lambda}  &\qquad\text{if \;\;} \displaystyle0<\dt\le\frac{2\eps}{\lambda+1} \;,\\ 
			&&or \\
%		&	{\color{blue}\displaystyle \Dt  \, \exp\left(-\frac{M\dt^\star}{\eps}\right) < \frac{\eps}{\lambda} }& {\color{blue} \qquad \text{if \;\;} \displaystyle 0<\dt^\star<\frac{2\eps}{\lambda+1}}\; ,
%		&	{\color{blue}\displaystyle \Dt  \, \exp\left(-\frac{M(2\eps - \lambda \dt)}{\eps}\right) < \frac{\eps}{\lambda} }& {\color{blue} \qquad \text{if \;\;} \displaystyle\frac{2\eps}{\lambda+1}<\dt<\frac{2\eps}{\lambda}}\; .
		&	\displaystyle \Dt  \, \exp\left(-\frac{M\dt^\star}{\eps}\right) < \frac{\eps}{\lambda} & \qquad \text{if \;\;} \displaystyle\frac{2\eps}{\lambda+1}<\dt<\frac{2\eps}{\lambda}\; .
		\end{array}
\end{equation*}
with $\dt^\star=2\varepsilon-\lambda \dt$. The range $2\eps/(\lambda+1) < \dt<2\eps/\lambda$ corresponds to $0<\dt^\star<2\varepsilon/(\lambda+1)$. This assumption is necessary to bound the distance of the fast variables from the centre manifold over the macrosteps.\\
\end{itemize}
%\end{assumptions}

\noindent
The global Lipschitz conditions can be relaxed to local Lipschitz conditions by the usual means.\\
% Note that for the weights (\ref{W_HMM0}) or (\ref{e.WPI}) we have that $\tilde{g}$ is bounded by $C_g$.\\

\noindent
We will establish bounds for the error $\mathcal{E}^n$ between the PI estimate $y^n$ and the solution of the full system $y_\eps(t^n)$
\begin{align*}
{\mathcal{E}}^n = |y_\eps(t^n)-y^n|\; .
\end{align*}
Our main result is provided by the following Theorem:
\medskip
\begin{theorem}
\label{theorem.main}
Given the assumptions ({\it A1})--({\it A8}), there exists a constant $C$ such that  on a fixed time interval $T$, for each $n$ such that $n\td\le T$, the error between the PI estimate and the exact solution of the full multiscale system (\ref{baseslow})-(\ref{basefast}) is given by
\begin{align*}
{\mathcal{E}}^n \le  
 \twoparts
{\quad \;\;\;\;C\left(
\td + \varepsilon
+\left(\frac{\varepsilon}{\td}+\eps+e^{-\frac{M\dt}{\varepsilon}} \right)|d^n| 
\right)}
%{C\left(
%\td + \left(1+\frac{\dt}{\dt^*}\right)\varepsilon
%+\left(\frac{\dt}{\dt^*}\frac{\varepsilon}{\td}+\frac{\dt}{\dt^*}\eps+e^{-\frac{M\dt^*}{\varepsilon}} %\right)|d^n| 
{C\frac{\dt}{\dt^*}\left(
\td + \varepsilon
+\left(\frac{\varepsilon}{\td}+\eps+e^{-\frac{M\dt^*}{\varepsilon}} \right)|d^n| 
\right)}\;\; ,
\end{align*}
where $d^n$ measures the maximal distance of the fast variables $x$ from the approximate centre manifold $x=\bar{f}(y):= \Lambda^{-1} f(y)$ over all macrosteps and is estimated by
\begin{align*}
\left|d^n\right| \leq&
\twoparts
{\left| x^0 - \bar{f}(y^0) \right|  + \frac{\eps L_f C_g(1 + \lambda)  \Dt }{\eps-\Dt\lambda e^{-\frac{M\dt}{\eps}}}}
{\left| x^0 - \bar{f}(y^0) \right|  + \frac{\eps L_f C_g(1 + \frac{\dt}{\dt^*}\lambda)  \Dt }{\eps-\Dt\lambda e^{-\frac{M\dt^*}{\eps}}}}
 \;\;,
\end{align*}
with $\dt^* = 2\eps - \lambda \dt$.
\end{theorem}

\medskip
Before proceeding to the proof of the theorem, it is worthwhile to interpret the bounds on $\mathcal{E}^n$. The term proportional to $\td=\Dt + M \dt$ reflects the first order convergence of the forward Euler numerical scheme used to propagate the micro- and macrosolver respectively. The terms proportional to the time scale parameter $\varepsilon$ represent the error made by the reduction as well as an additional error incurred during the drift of the slow variable over the microsteps. The term proportional to $\exp(-M\dt/\varepsilon)|d^n|$ measures the exponential decay of the fast variables towards the slow manifold.
%{ \color{blue} At the larger values of $\dt$ with $\frac{2\eps}{\lambda+1}<\dt<\frac{2\eps}{\lambda}$, the fast variables converge towards the slow manifold with rate $\exp(-M\dt^*/\eps)$, which is a product of the forward Euler scheme used for the microsolver.}

%%%%%%%%%%%%%%%%%%%%%%%%%%%%%%%%%%%%%%%

\subsection{Error Analysis}
We split the error into two parts according to
\begin{align*}
{\mathcal{E}}^n &= |y_\eps(t^n)-y^n|\\
&\le |y_\eps(t^n)-Y(t^n)| + |y^n-Y(t^n)|
\; ,
\end{align*}
where the first term describes the error between the exact solution of the full system (\ref{baseslow})-(\ref{basefast}) and the reduced slow system (\ref{e.CMT}), which we label {\emph{reduction error}}, and the second term the error between PI and the exact solution of the reduced slow system (\ref{e.CMT}), which we denote by {\emph{discretization error}}. We will bound the two terms separately in the following. 

%%%%%%%%%%%%%%%%%%%%%%%%%%%%%%%%%%%%%%%

\subsection{Reduction error}
Defining the error between the exact solution of the full system (\ref{baseslow})-(\ref{basefast}) and the reduced slow system (\ref{e.CMT})
\begin{align}
|E_c^n| = |y_\eps(t^n)-Y(t^n)|\; ,
\end{align}
and setting the initial conditions close to the slow manifold with $y_\eps(0)=Y(0)+c_{0,y}\,\eps$ and $x_\eps(0)=f(y_\eps(0))+c_{0,x}$, we can formulate the following theorem:
\medskip
\begin{theorem}
\label{theorem.Ec}
Given the assumptions ({\it A1})--({\it A3}), there exists a constant $C_1$ such that on a fixed time interval $T$, for each $t^n\le T$, the error between exact solutions of the reduced and the full system is bounded by
\begin{align*}
|E_c^n| \le C_1 \varepsilon \; ,
\end{align*}
with 
\[
C_1=
{\rm{max}}\left(|y_\eps(0)-Y(0)|,L_fL_gC_gt^n,L_g|d_\eps(0)|) \right) 
e^{L_g\left(1+L_f\right)t^n}
\; ,
\]
where $d_\eps=x_\eps-f(y_\eps)$ measures the distance of the fast variables from the approximate slow manifold.
\end{theorem}
\begin{proof}
The proof is standard in centre manifold theory and is included for completeness. Using assumptions ({\it A1})--({\it A2}) on the Lipschitz continuity of $f$ and $g$ we estimate
\begin{align}
\frac{d}{dt}|y_\eps(t)-Y(t)| 
&= |g(x_\eps,y_\eps)-g(\bar{f}(Y),Y)|
\nonumber
\\
&\le |g(x_\eps,y_\eps)-g(x_\eps,Y)|+|g(x_\eps,Y)-g(\bar{f}(Y),Y)|
\nonumber
\\
&\le L_g|y_\eps-Y|+L_g|x_\eps-\bar{f}(Y)|
\nonumber
\\
&\le L_g |y_\eps-Y| + L_g |x_\eps-\bar{f}(y_\eps)| + L_g |\bar{f}(y_\eps)-\bar{f}(Y)|
\nonumber
\\
&\le L_g\left(1+L_f\right) |y_\eps-Y| + L_g |x_\eps-\bar{f}(y_\eps)|\; ,
\label{e.yrederr}
\end{align}
To estimate the second term, we differentiate the distance of the fast variable to the slow manifold $d_\eps=x_\eps-\bar{f}(y_\eps)$ with respect to time, using (\ref{basefast}), as
\begin{align*}
\dot{d}_\eps = -\frac{\Lambda}{\eps}d_\eps - D\bar{f}(y_\eps)\,g(x_\eps,y_\eps)\; ,
\end{align*}
where $\D \bar{f}$ denotes the Jacobian matrix of $\bar{f}$. Integrating and using the boundedness assumptions ({\it A4}) on $g$ and the Lipschitz continuity assumption ({\it A1}) on $f$, we readily estimate  
\begin{align}
|d_\eps(t)| &\le \exp{(-\frac{\Lambda}{\eps}t)}\left|d_\eps(0)\right| + \left|\int_0^t \exp{(-\frac{\Lambda}{\eps}(t-s))}D\bar{f}(y_\eps(s))\,g(x_\eps(s),y_\eps(s))\, ds\right|
\nonumber
\\
&\le e^{-\frac{t}{\eps}}|d_\eps(0)| + \sup_{|g| \le C_g}|Df\,g| \int_0^t e^{-\frac{t-s}{\eps}}\, ds\, 
\nonumber
\\
%&\le e^{-\frac{t}{\eps}}|d_\eps(0)| + C_g  \sup_{|\xi|\le1}|Df\, \xi|\, \int_0^t e^{-\frac{t-s}{\eps}}\, ds
%\nonumber
%\\
& \le 
e^{-\frac{t}{\eps}}|d_\eps(0)| + \eps L_f C_g
\label{e.x0est}
\; ,
\end{align}
which signifies the exponential attraction of the fast dynamics towards the slow centre manifold. Substituting (\ref{e.x0est}) into (\ref{e.yrederr}) yields upon application of the Gronwall lemma
\begin{align*}
|y_\eps(t)-Y(t)| 
&\le 
e^{L_g\left(1+L_f\right)t}
\left(
|y_\eps(0)-Y(0)| + \eps L_fL_gC_gt+\eps L_g|d_\eps(0)|
\right) 
\; ,
\end{align*}
which immediately yields the desired result using the initial condition $d_\eps(0)=c_{0,x}$ and $y_\eps(0)=Y(0)+c_{0,y}\eps$. 
\end{proof}

%%%%%%%%%%%%%%%%%%%%%%%%%%%%%%%%%%%%%%%

\subsection{Discretization error}
In order to bound the discretization error $|y^n-Y(t^n)|$ we first estimate how close the fast variables remain to the centre manifold during the application of the microsolver. Defining the deviation of the PI approximation of the fast variables from the approximate centre manifold at time $t^n+m\dt$ 
\begin{align}
d^{n,m} = x^{n,m}-\bar{f}(y^{n,m})\; ,
\end{align}
we formulate the following Lemma which constitutes the discrete time version of (\ref{e.x0est}):
\medskip
\begin{lemma}
\label{lemma_dnm}
Given assumptions ({\it A1}), ({\it A4}) and ({\it A6}), the error between the fast variables and the approximate centre manifold during the application of the PI microsolver is bounded for all $0\leq m\leq M$ by
\begin{align*}
|d^{n,m}|\le \twoparts
{\left( 1-\frac{\dt}{\varepsilon}\right)^m |d^{n,0}|  +\varepsilon L_fC_g}
{\left( 1-\frac{\dt^*}{\varepsilon}\right)^m |d^{n,0}|  +\frac{\dt}{\dt^*}\varepsilon  L_fC_g}
\;\; .
\end{align*}
\end{lemma}
The first term is a manifestation of the exponential decay of the fast variables towards the slow centre manifold along their stable eigendirection. The second term proportional to $\eps$ describes, as we will see below, the cumulative effect of  changes of the slow variables $y$ during the microsteps causing a departure from the centre manifold for nonconstant $\bar{f}(y)$.

\begin{figure}
\includegraphics[width=0.8\textwidth]{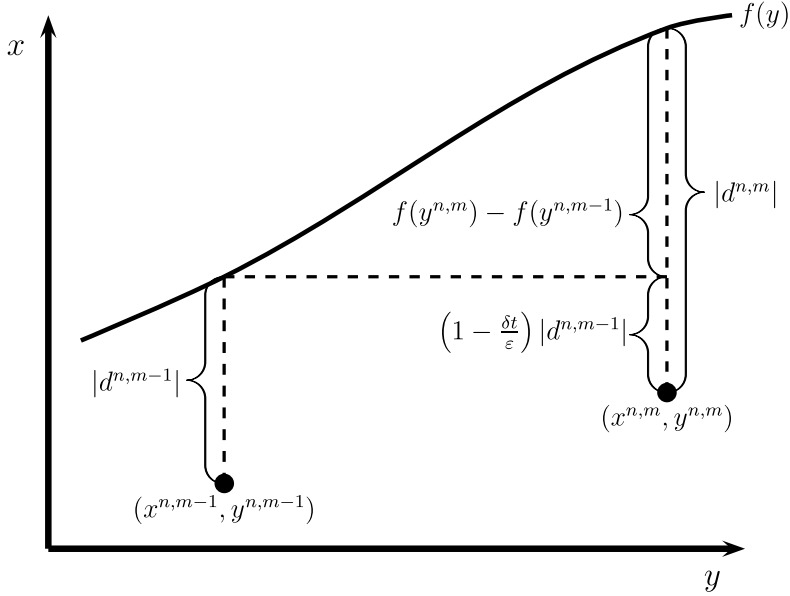}
\caption{Illustration of equation \eqref{e.dnmonestep} for $x_\eps \in \mathbb{R}$, so that $\Lambda=1$. The two dots show one microstep in PI.}
\label{fig:microstep}
\end{figure}

\begin{proof}
Using the definition of the PI microsolver (\ref{micro_x}), a recursive relationship for $d^{n,m}$ is readily established as
\begin{align}
d^{n,m}&=  x^{n,m} - \bar{f}(y^{n,m})
\nonumber
\\
%&= x^{n,m-1} - \frac{ \dt}{ \varepsilon}(\Lambda \, x^{n,m-1} - {f}(y^{n,m-1}) ) - \bar{f}(y^{n,m})
%\nonumber
%\\
& = \left(\I-\frac{\dt}{\varepsilon}\Lambda\right) d^{n,m-1} - \left( \bar{f}(y^{n,m}) - \bar{f}(y^{n,m-1})\right)
\; .
\label{e.dnmonestep}
\end{align} 
Let us pause here for a moment to discuss the terms appearing in this equation. The first term illustrates that at each microstep the distance between the $i$th fast variable and the approximate slow manifold decreases by a factor of $1-\dt \, \lambda_{ii}/\eps$. The second term represents the change in the approximate slow manifold during the drift of the slow variable $y$ in one microstep. This is illustrated in Figure~\ref{fig:microstep}.\\ 

\noindent
The cumulative effect of $M$ microsteps on $d^{n,m}$ is easily obtained from (\ref{e.dnmonestep}) as 
\begin{align}
d^{n,m} 
= \left(\I -\frac{\dt}{\varepsilon}\Lambda\right)^m d^{n,0} - \sum_{k=0}^{m-1}\left(\I  - \frac{\dt}{\varepsilon}\Lambda \right)^k \left(\bar{f}(y^{n,m-k})-\bar{f}(y^{n,m-1-k})\right)\; .
\label{e.dnm_pure}
\end{align} 
Taking absolute values and using the Lipschitz condition on $f$ {\em{(A1)}}, the bound on the slow vectorfield $g$ {\em{(A4)}} and (\ref{micro_y}), the cumulative effect of $M$ microsteps on the distance $|d^{n,m}|$ of the fast variables to the approximate centre manifold is
\begin{align*}
|d^{n,m}| &\leq \left|\I -\frac{\dt}{\varepsilon}\Lambda\right|^m |d^{n,0}| + L_f C_g \dt \sum_{k=0}^{m-1} \left|\I -\frac{\dt}{\varepsilon}\Lambda\right|^k
\\
&= \left|\I -\frac{\dt}{\varepsilon}\Lambda\right|^m |d^{n,0}| + L_f C_g \dt \frac{1 -  |\I -\frac{\dt}{\varepsilon}\Lambda|^m}{1 - |\I -\frac{\dt}{\varepsilon}\Lambda| }\;.
\end{align*}
To avoid divergence from the centre manifold we require
\[
\left|\I -\frac{\dt}{\varepsilon}\Lambda\right| { = \max\left\{ 1-\frac{\dt}{\eps}, \lambda \frac{\dt}{\eps} - 1\right\} } < 1\; ,
\]
implying the standard condition $0<\dt<2{\eps}/{\lambda}$. For ${ 0<\dt \leq \frac{2\eps}{\lambda+1} }$, $\left|\I -\frac{\dt}{\varepsilon}\Lambda\right| = 1-\frac{\dt}{\eps}$ and we obtain the bound
 \begin{align}
|d^{n,m}|
\leq
\left(1-\frac{\dt}{\varepsilon}\right)^m |d^{n,0}| + \varepsilon L_f C_g \left[1 -  \left(1-\frac{\dt}{\varepsilon}\right)^m\right]\; .
\label{e.dnm_dt}
\end{align}
Furthermore for $\frac{2\eps}{\lambda+1} < \dt < \frac{2\eps}{\lambda}$, introducing $\dt^\star=2\eps-\lambda\dt$, $\left|\I -\frac{\dt}{\varepsilon}\Lambda\right| = \lambda \frac{\dt}{\eps} - 1 = 1-\frac{\dt^*}{\eps}$, yielding
 \begin{align}
|d^{n,m}|
\leq
\left(1-\frac{\dt^*}{\varepsilon}\right)^m |d^{n,0}| + \frac{\dt}{\dt^*} \varepsilon L_f C_g \left[1 -  \left(1-\frac{\dt^*}{\varepsilon}\right)^m\right]\; .
\label{e.dnm_dt*}
\end{align}
Lemma~\ref{lemma_dnm} follows from \eqref{e.dnm_dt}--\eqref{e.dnm_dt*}. Note that the term proportional to $\eps$ was generated by a sum of error terms in $\dt$ relating to the change in the manifold during the drift of the slow variable $y^{n,m}$ over each microstep.
\end{proof}
\medskip
\begin{remark}
 In the limit $\dt^\star\to 0$ (i.e. $\dt\to 2\eps/\lambda$) the term $\left[1 -  \left(1-\frac{\dt^*}{\varepsilon}\right)^m\right]$ in \eqref{e.dnm_dt*}, which is neglected in Lemma~\ref{lemma_dnm}, is crucial to obtain a finite estimate $|d^{n,m}| \le |d^{n,0}| +m\dt L_f C_g$.
\end{remark}
\medskip
\begin{remark} 
In the case $x_\eps\in \R$ with $\Lambda=1$, a direct estimation of (\ref{e.dnmonestep}) implies that for the particular choice $\dt=\eps$ we have $|d^{n,m}|\leq L_fC_g\dt$, i.e. there is no dependence on initial conditions for the fast variable; this is a peculiarity of the underlying forward Euler scheme used (cf. (\ref{micro_x})) where for $\dt=\eps$ we have $x^{n,m+1}=f(y^{n,m})$. 
%For the particular choice $\dt=\eps$ we obtain $|d^{n,m}|\leq L_fC_g\dt$ by direct estimation of (\ref{e.dnmonestep}), i.e. there is no dependence on initial conditions for the fast %variable; this is a peculiarity of the underlying forward Euler scheme used (cf. (\ref{micro_x})) where for $\dt=\eps$ we have $x^{n,m+1}=f(y^{n,m})$. 
\end{remark}
%\medskip
%{\color{blue}
%\begin{remark}
%The range of micro time steps $\dt<2 \eps/(\lambda+1)$ provides computational advantage in systems with $\lambda\gg1$. Lemma~\ref{lemma_dnm} assures that does not need to resolve the fastest time scale employing a computationally more stringent $\dt<\eps/\lambda$.
%\end{remark} }

\medskip
We have bounded the distance of the PI approximation of the fast variables $x^{n,m}$ from the slow centre manifold $\bar{f}(y^{n,m})$ over the microsteps. We now establish bounds on the distance of the PI approximation of the slow variables $y^{n,m}$ from the solution of the reduced dynamics on the centre manifold over the microsteps. 
%For this purpose we study the reduced dynamics resolved on the microsteps. 
We therefore introduce the auxiliary time-continuous system
\begin{align}
\label{Yn.o}  \dot{Y}^{(n)}(s) &= G(Y^{(n)}(s)) \;  \\
\label{Yn.i}  Y^{(n)}(0) &= y^n \; ,
\end{align}
which resolves the reduced slow dynamics \eqref{e.CMT} initialised after each macrostep at time $t^n = n\td$ with the PI approximation of the slow variables $y^n$. Further we introduce the discrete Euler approximation of the reduced slow dynamics \eqref{e.CMT}
\begin{align}
\label{phi.a} 
 \phieta^m &= \phieta^{m-1} + \dt\,G(\phieta^{m-1})\\
\label{phi.b}
\phieta^0 &= \eta \; ,
\end{align}
for arbitrary initial conditions $\eta$. Note that for the particular choice $\eta = y^n$ (\ref{phi.a})-(\ref{phi.b}) corresponds to an Euler discretization of the system \eqref{Yn.o}-\eqref{Yn.i}.

\medskip
\begin{lemma} Assuming ({\it A1}),({\it A2}),({\it A5}) and ({\it A7}), the numerical estimate $y^{n,m}$ of the slow variable is close to $\phiyn^m$ with
\begin{align*}
%|y^{n,m} - \phiyn^m| \leq 3 L_g \eps|d^{n,0}| + \mathcal{O}(m\dt^2, m\dt\eps) \; .
|y^{n,m} - \phiyn^m| \leq 
\twoparts
{\;\;\;\;\;\;3 L_g \eps|d^{n,0}| +  2 \frac{L_fC_g\eps}{L_f+1} \;\; + \mathcal{O}(m\dt^2)}
{\left( 3 L_g \eps|d^{n,0}| +  2 \frac{L_fC_g\eps}{L_f+1}\right)\frac{\dt}{\dt^*}  + \mathcal{O}(m\dt^2)} \;\; .
\end{align*}
\label{phi.bound}
\end{lemma}
\begin{proof}
Employing the auxiliary system \eqref{Yn.o}-\eqref{Yn.i}, we split the error into
\begin{align}
\label{lemma.aux}
|y^{n,m} - \phiyn^m| \leq |y^{n,m} - Y^{(n)}(m\dt)| + |Y^{(n)}(m\dt) - \phiyn^m| \;.
\end{align}
The second term above can be readily bounded by the standard proof for Euler convergence (see for example \cite{Iserles}). The first term is bounded by a slight variation of the same proof. For completeness we provide the proof of the bounds for both terms, beginning with the bound on the second term. By Taylor expanding we write upon using the auxiliary dynamics (\ref{Yn.o})
\begin{align}
\label{lemma.T}
Y^{(n)}(m\dt) = Y^{(n)}\big((m-1)\dt\big) + \dt\,G\big(Y^{(n)}\big((m-1)\dt\big)\big) + \mathcal{O}(\dt^2) \;.
\end{align}
Note that assumptions ({\it A1}) and ({\it A2}) imply a Lipschitz constant $L_G\leq L_g(1+L_f)$ for the reduced vectorfield $G$. Using the Euler discretization \eqref{phi.a} of the reduced auxiliary system \eqref{Yn.o}-\eqref{Yn.i} with initial condition $y^n$, and employing assumption ({\it A5}) with $C^* = \sup|\ddot{Y}(t)|$ to bound the ${\mathcal{O}}(\dt^2)$-term in (\ref{lemma.T}), we obtain
\begin{align*}
|Y^{(n)}(m\dt) - \phiyn^m| &\leq \left|Y^{(n)}\big(\left(m-1\right)\dt\big) - \phiyn^{m-1}\right|\\
&+ \dt\left| G\big(Y^{(n)}\big((m-1)\dt\big)\big) - G(\phiyn^{m-1}) \right| + C^*\dt^2 
\\
&\leq (1 + L_G\dt) \left|Y^{(n)}\big(\left(m-1\right)\dt\big) - \phiyn^{m-1}\right| + C^*\dt^2
\\
&\leq C^*\dt^2 \sum_{k=0}^{m-1} (1+L_G\dt)^k \; ,
\end{align*}
since we initialize with $Y^{(n)}(0) = \phiyn^{0}$. Evaluating the sum we find
\begin{align}
\nonumber |Y^{(n)}(m\dt) - \phiyn^{m}| &\leq C^*\dt^2 \frac{(1+L_G\dt)^m-1}{L_G\dt}
\\
%\nonumber&\leq C^*\dt^2 \frac{e^{L_Gm\dt}-1}{L_G\dt}
%\\
%\nonumber&\leq C^*\dt^2 \frac{2 L_G m \dt}{L_G\dt}
%\\
%\label{lemma.close}&= 2 C^* m \dt^2 \;,
&\le 2 C^* m \dt^2 \;,
\label{lemma.close}
\end{align}
where assumption ({\it A7}), that $L_GM\dt\le 1$, was used to bound $e^{L_Gm\dt}-1\leq 2 L_G m \dt$.\\
To bound the first term in \eqref{lemma.aux}, $|y^{n,m} - Y^{(n)}(m\dt)|$, we analogously obtain 
\begin{align}
\nonumber \left|y^{n,m} - Y^{(n)}(m\dt)\right| 
%\leq&  \left|y^{n,m-1} - Y^{(n)}\big((m-1)\dt\big)\right| \\
%&+ \dt\left| g(x^{n,m-1},y^{n,m-1}) - G\big(Y^{(n)}\big((m-1)\dt\big)\big)\right| + C^*\dt^2
%\\
&\leq  \left|y^{n,m-1} - Y^{(n)}\big((m-1)\dt\big)\right| 
\\
\nonumber &+ \dt\left| g(x^{n,m-1},y^{n,m-1}) - G(y^{n,m-1}) \right|
\\
\nonumber &+ \dt \left| G(y^{n,m-1}) -  G\big(Y^{(n)}\big((m-1)\dt\big)\big)\right| + C^*\dt^2
\\
%\leq& (1+L_G \dt) \left|y^{n,m-1} - Y^{(n)}\big((m-1)\dt\big)\right| 
%\\
%&+ \dt\left| g(x^{n,m-1},y^{n,m-1}) - G(y^{n,m-1}) \right| + C^*\dt^2
%\\
%=& (1 + L_G \dt) \left|y^{n,m-1} - Y^{(n)}\big((m-1)\dt\big)\right| 
%\\
%&+ \dt\left|  g(x^{n,m-1},y^{n,m-1}) - g\left(\bar{f}(y^{n,m-1}),y^{n,m-1}\right)  \right| + C^*\dt^2
%\\
%\leq& (1 + L_G \dt) \left|y^{n,m-1} - Y^{(n)}\big((m-1)\dt\big)\right| 
%\\
%&+ L_g \dt\left| \bar{f}(y^{n,m-1}) - x^{n,m-1} \right| + C^*\dt^2
%\\
%=
\nonumber &\leq  (1 + L_G \dt) \left|y^{n,m-1} - Y^{(n)}\big((m-1)\dt\big)\right| 
\\
\label{alt.aux} &+ L_g \dt\left| d^{n,m-1} \right| + C^*\dt^2 \; .
\end{align}
Employing Lemma~\ref{lemma_dnm} with $0<\dt<\frac{2\eps}{\lambda+1}$, we obtain
%\leq& (1 + L_G \dt)\left|y^{n,m-1} - Y^{(n)}\big((m-1)\dt\big)\right| 
%\\
%&+ L_g\dt \left[\left(1-\frac{\dt}{\eps}\right)^{m-1}|d^{n,0}| + L_f C_g\eps\right]+ C^*\dt^2 
%\\
\begin{align*}
 \left|y^{n,m} - Y^{(n)}(m\dt)\right| &\leq \dt \,L_g |d^{n,0}| \sum_{k=0}^{m-1}(1+L_G\dt)^k \left(1-\frac{\dt}{\eps}\right)^{m-1-k}
\\
 &+ \left[L_g L_f C_g\eps\dt+ C^*\dt^2 \right] \sum_{k=0}^{m-1} (1+L_G\dt)^k \;,
\end{align*}
since $y^{n,0}= y^n = Y^{(n)}(0)$. Evaluating the sums and using assumption ({\it A7}) yields
\begin{align*}
\left|y^{n,m} - Y^{(n)}(m\dt)\right| \leq& L_g \dt |d^{n,0}|\frac{(1+L_G\dt)^m-\left(1-\frac{\dt}{\eps}\right)^m}{L_G\dt + \frac{\dt}{\eps}} + 2C^* m \dt^2 
\\
&+ 2 L_gL_fC_g\eps m \dt
\\
%=& L_g \eps |d^{n,0}|\frac{(1+L_G\dt)^m-\left(1-\frac{\dt}{\eps}\right)^m}{L_G\eps + 1} + 2C^* m \dt^2 
%\\
%&+ 2 L_gL_fC_g\eps \frac{L_f + 1}{L_f+1}  m \dt
%\\
&\leq L_g \eps |d^{n,0}|(1+2L_Gm\dt) + 2C^* m \dt^2 + 2 \frac{L_fC_g\eps}{L_f+1} 
\\
&\leq 3 L_g \eps |d^{n,0}| + 2C^* m \dt^2 +  2 \frac{L_fC_g\eps}{L_f+1} \;  ,
\end{align*}
which concludes the proof of the lemma for $0<\dt<\frac{2\eps}{\lambda+1}$. Analogously for $\frac{2\eps}{\lambda+1}<\dt<\frac{2\eps}{\lambda}$, Lemma~\ref{lemma_dnm} is used to estimate \eqref{alt.aux}
\begin{align*}
\left|y^{n,m} - Y^{(n)}(m\dt)\right|  &\leq \dt \,L_g |d^{n,0}| \sum_{k=0}^{m-1}(1+L_G\dt)^k \left(1-\frac{\dt^*}{\eps}\right)^{m-1-k}
\\
 &+ \left[L_g L_f C_g\eps\frac{\dt^2}{\dt^*}+ C^*\dt^2 \right] \sum_{k=0}^{m-1} (1+L_G\dt)^k \;
 \\
 &\leq  L_g \dt |d^{n,0}|\frac{(1+L_G\dt)^m-\left(1-\frac{\dt^*}{\eps}\right)^m}{L_G\dt + \frac{\dt^*}{\eps}} + 2C^* m \dt^2 
\\
&+ 2 L_gL_fC_g\eps \, m \frac{\dt^2}{\dt^*}
\\
&\leq3L_g \eps \frac{\dt}{\dt^*} |d^{n,0}| + 2C^* m \dt^2 + 2 \frac{L_fC_g\eps}{L_f+1} \frac{\dt}{\dt^*} \;\; .
\end{align*} 
\end{proof}

Lemma~\ref{phi.bound} established bounds on the PI approximation of the slow variables and an Euler approximation of the reduced system during the microsteps. We will now establish bounds on the PI approximation of the slow variables and an Euler approximation of the reduced system during one macrostep. In order to do so we introduce the vector field
\begin{align}
\label{define_G_twiddle}\tilde{G}(\eta) := \sum_{m=0}^M W_m G(\phieta^m) \; ,
\end{align}
with weights $W_m$ given by \eqref{e.WPI} as in PI. We now proceed to bound $|\tilde{g}(x^n,y^n) - \tilde{G}(y^n)|$.

\medskip
\begin{lemma} \label{g-G}
Assuming ({\it A1})-({\it A6}), the auxiliary vector field $\tilde{G}$ is close to the vector field $\tilde{g}$ of PI, with
\begin{align*}
|\tilde{g}(x^n,y^n)-\tilde{G}(y^n)| \leq& 
\twoparts
{\begin{aligned}&L_g \bigg(\frac{\varepsilon}{\td} + 3L_g(1+L_f)\eps\bigg)|d^{n,0}| 
\\&+ L_g e^{-\frac{M\dt}{\varepsilon}}|d^{n,0}| + 3L_g L_f C_g \varepsilon\end{aligned}\;\;\vspace{10pt}}
{\begin{aligned}&L_g \bigg(\frac{\dt}{\dt^*}\frac{\varepsilon}{\td} + 3\frac{\dt}{\dt^*}L_g(1+L_f)\eps \bigg) |d^{n,0}| 
\\&+ L_g e^{-\frac{M\dt^*}{\varepsilon}} |d^{n,0}|+ 3 \frac{\dt}{\dt^*}  L_g L_f C_g \eps
\end{aligned}\;\;} \;\; .
\end{align*}
\end{lemma}
\begin{proof}
\noindent
Employing the Lipschitz condition ({\it{A1}}) on $f$ and ({\it{A2}}) on $g$, we write
\begin{align}
\nonumber |\tilde{g}(x^n,y^n)-\tilde{G}(y^n)| =&\left|  \sum_{m=0}^M W_m \Big[ g(x^{n,m},y^{n,m}) - G(\phiyn^m) \Big]  \right|
\\
%=&\left|  \sum_{m=0}^M W_m \Big[  g(x^{n,m},y^{n,m}) - g(\bar{f}(\phiyn^m),\phiyn^m) \Big] \right|
%\\
\nonumber &\leq \sum_{m=0}^M W_m L_g  \Big[ |x^{n,m} - \bar{f}(\phiyn^m)| + |y^{n,m} - \phiyn^m|  \Big]
\\
\nonumber &\leq \sum_{m=0}^M W_m L_g  \Big[ |x^{n,m} - \bar{f}(y^{n,m})| + |\bar{f}(y^{n,m}) - \bar{f}(\phiyn^m)| \\
\nonumber &\hspace{50pt}+ |y^{n,m} - \phiyn^m|  \Big]
\\ 
\label{g-G.aux}&\leq \sum_{m=0}^M W_m L_g  \Big[ |d^{n,m}| + (1+L_f)|y^{n,m} - \phiyn^m|  \Big] \; .
\end{align}
Using Lemmas~\ref{lemma_dnm}~and~\ref{phi.bound} for the case when $0<\dt<\frac{2\eps}{\lambda+1}$, we expand 
\begin{align}
\nonumber |\tilde{g}(x^n,y^n)-\tilde{G}(y^n)| \leq
&
\sum_{m=0}^M W_m \left[ L_g\left(1-\frac{\dt}{\eps}\right)^m |d^{n,0}|\right]+ 3 L_g L_f C_g \eps 
\\
\label{E_s}&\qquad\qquad+ 3L_g^2(1+L_f)\eps|d^{n,0}| \; .
\end{align}

\noindent Inserting the weights $W_m$ (\ref{e.WPI}) which characterize PI, and evaluating the geometric series, we obtain
\begin{align*}
|\tilde{g}(x^n,y^n)-\tilde{G}(y^n)| &\leq \sum_{m=0}^{M-1} \frac{\dt}{\td}  L_g\left(1-\frac{\dt}{\eps}\right)^m |d^{n,0}| + \frac{\Dt}{\td}  L_g\left(1-\frac{\dt}{\eps}\right)^M |d^{n,0}|
\\
& + 3 L_g L_f C_g \eps + 3L_g^2(1+L_f)\eps|d^{n,0}|
\\
%\leq& \frac{\eps}{\td} L_g \left(1-\left(1-\frac{\dt}{\eps}\right)^M\right) |d^{n,0}| + \frac{\Dt}{\td}  L_g e^{-\frac{M\dt}{\eps}} |d^{n,0}|
%\\
%& + 3 L_g L_f C_g \eps + 3L_g^2(1+L_f)\eps|d^{n,0}|
%\\
&\leq L_g \left(\frac{\varepsilon}{\td} + 3L_g(1+L_f)\eps + e^{-\frac{M\dt}{\varepsilon}} \right)|d^{n,0}|+ 3 L_g L_f C_g \varepsilon \; ,
\end{align*}
which concludes the proof for that case. Using Lemmas~\ref{lemma_dnm}~and~\ref{phi.bound} for $ \frac{2\eps}{\lambda+1} < \dt < \frac{2\eps}{\lambda}$ in \eqref{g-G.aux}, we obtain analogously
\begin{align}
\nonumber |\tilde{g}(x^n,y^n)-\tilde{G}(y^n)| &\leq
\sum_{m=0}^M W_m \left[ L_g\left(1-\frac{\dt^*}{\eps}\right)^m |d^{n,0}|\right]+ {3 \frac{\dt}{\dt^*}L_g L_f C_g \eps  }
\\
\label{E_s.aux}&\qquad\qquad+ 3\frac{\dt}{\dt^*}L_g^2(1+L_f)\eps|d^{n,0}| \; 
\\
\nonumber  &\leq L_g \left(\frac{\dt}{\dt^*} \frac{\varepsilon}{\td}+ 3\frac{\dt}{\dt^*} L_g(1+L_f)\eps+ e^{-\frac{M\dt^*}{\varepsilon}} \right)|d^{n,0}|+ 3 \frac{\dt}{\dt^*} L_g L_f C_g \varepsilon \; .
\end{align}
\end{proof}
\begin{remark}
The above estimates \eqref{E_s} and \eqref{E_s.aux} not only hold for PI, but also for the seamless formulation of HMM, since they are independent of the particular choice for the weights $W_m$ and of the method of reinitialization of the fast variables $x^n$.
\end{remark}

\medskip
\noindent 
We have established in Lemma~\ref{g-G} that the vector field $\tilde{G}$ given by \eqref{define_G_twiddle} is close to the PI approximation $\tilde{g}$ of the slow dynamics over a time step of $\td$. In the following Lemma we demonstrate that $\tilde{G}$ can be used to step forward the reduced slow variables in a macrostep.

\medskip
\begin{lemma}
\label{lemma_Gtwiddle} 
$\tilde{G}(Y(t^n))$ provides a numerical estimate of the reduced slow vector field with
\begin{align*}
Y(t^{n+1}) - Y(t^{n}) = \td \tilde{G}(Y(t^{n})) + \mathcal{O}(\td^2) \;,
\end{align*}
where the error term $\mathcal{O}(\td^2)$ is bounded by $2C^*\td^2$.
\end{lemma}
\begin{proof}
By Taylor expanding we write
\begin{align*}
Y(t^{n+1}) - Y(t^{n}) 
%&= \td  {G}(Y(t^{n})) + \mathcal{O}(\td^2) 
%\\
&= \td  \tilde{G}(Y(t^{n})) + \td  \left({G}(Y(t^{n})) -  \tilde{G}(Y(t^{n}))\right) + \mathcal{O}(\td^2) \; ,
\end{align*}
where according to assumption ({\it{A5}}) the ${\mathcal{O}}(\td^2)$-term is bounded by $C^\star \td^2$. We now bound
\begin{align*}
\left|{G}(Y(t^{n})) -  \tilde{G}(Y(t^{n}))\right| &= \left| {G}(Y(t^{n})) -  \sum_{m=0}^M W_m {G}(\phiYn^m) \right|
\\
%&= \left| \sum_{m=0}^M W_m\left( {G}(Y(t^{n})) - {G}(\phiYn^m) \right) \right|
%\\
&\leq \sum_{m=0}^M W_m L_G \left| Y(t^{n}) - \phiYn^m \right| 
\\
%&\leq \sum_{m=0}^M W_m L_G C_G m \dt
%\\
%&\leq L_G C_G M \dt 
%\\
&\leq L_G C_G \td \;,
\end{align*}
where $C_G \leq C_g$ is the maximum of $|G|$. Noting that $ C^* = \sup|\ddot{Y}| = \sup|DG(Y)\, G(Y)| = L_G C_G$ completes the proof.
\end{proof}

\noindent
Using Lemmas~\ref{lemma_dnm}, \ref{phi.bound}, \ref{g-G} and \ref{lemma_Gtwiddle} we now proceed to bound the discretization error
\begin{align}
|E_d^n| = |y^n-Y(t^n)|\; ,
\end{align}
and formulate the following theorem.
\medskip
\begin{theorem}
\label{theorem.Ed}
Given the assumptions ({\it A1})--({\it A8}), there exists a constant $C_2$ such that on a fixed time interval $T$, for each $n\Dt\le T$, the error between the solution of the projective integration scheme and the exact solutions of the reduced system is bounded by
\begin{align*}
E_d^n \le  
 \twoparts
{\quad \;\;\;\;C\left(
\td + \varepsilon
+\left(\frac{\varepsilon}{\td}+\eps+e^{-\frac{M\dt}{\varepsilon}} \right)|d^n| 
\right)}
{C\frac{\dt}{\dt^*}\left(
\td + \varepsilon
+\left(\frac{\varepsilon}{\td}+\eps+e^{-\frac{M\dt^*}{\varepsilon}} \right)|d^n| 
\right)}\;\; ,
\end{align*}
where $\left|d^n\right| := {\rm max}_{0\le i\le n-1}|d^{i,0}|$.
\end{theorem}
\medskip
\begin{remark}
The error estimate involves the well known exponential decay of the fast variable towards the approximate centre manifold, leading to a loss of memory of the fast initial condition; PI, however, involves an additional error term proportional to the maximal distance of the fast variable to the approximate centre manifold which involves no memory loss with $M\to \infty$. 
\end{remark}

\begin{proof} 
We estimate the difference $E_d^n=y^n-Y(t^n)$, applying Lemma~\ref{lemma_Gtwiddle},
\begin{align}
E_d^{n}-E_d^{n-1} &= y^{n} - y^{n-1} - \left(Y(t^{n}) - Y(t^{n-1})\right)
\nonumber
\\
&=  \tilde{g}(x^{n-1},y^{n-1})\td -  \tilde{G}\left(Y(t^{n-1})\right)\td + {\mathcal{O}}(\td^2)
\nonumber
\\ \vspace{10pt}
\label{E_n_diff} &=\left(\tilde{G}(y^{n-1}) - \tilde{G}(Y(t^{n-1}))\right)\td + \left(\tilde{g}(x^{n-1},y^{n-1}) - \tilde{G}(y^{n-1})\right)\td +{\mathcal{O}}(\td^2)
\nonumber
\\
&=  \LE^{n-1} E_d^{n-1}\td  +\alpha_{n-1} \td \; .
\end{align}
We used the mean value theorem for vector-valued functions to introduce 
\begin{align}
\LE^n&:=\int_0^1 \! \D \tilde{G}\big(Y(t^n)+\theta(y^n-Y(t^n))\big) \, d\theta \; ,
\nonumber
\end{align}
where $\D \tilde{G}$ is the Jacobian matrix of $\tilde{G}$, and we set
\begin{align*}
\alpha_n &:= \tilde{g}(x^n,y^n) - \tilde{G}(y^n) +{\mathcal{O}}(\td) \; ,
\end{align*}
where according to Lemma~\ref{lemma_Gtwiddle} we have
\begin{align}
\label{e.alphan}
|\alpha_n| \leq |\tilde{g}(x^n,y^n) - \tilde{G}(y^n)| +2C^*\td \; .
\end{align}
Taking absolute values we obtain
\begin{align*}
\left|E_d^n\right| 
%&= \left|(I + \td \LE^{n-1})E_d^{n-1} + \alpha_{n-1} \td \right|
%\\
&\leq \left(1+\td ||\LE^{n-1}||\right)\left|E_d^{n-1}\right|+ |\alpha_{n-1}| \td
\; .
\end{align*}
Employing assumption $({\it A4})$ on the boundedness of $g$, we define
\begin{align*}
{\hat{\alpha}} = \max_{0\le i \le n-1} | \alpha_i |\; ,
\end{align*}
and using $({\it A5})$ on the boundedness of $G$ it is easy to show that
\begin{align*}
L_G = \max_{0\le i\le n-1}||\LE^i|| \; .
\end{align*}
We obtain, evaluating the geometric series,
\begin{align}
\nonumber \left|E_d^n\right| 
%&\leq \left(1+\td L_G\right)\left|E_d^{n-1}\right| + {\hat{\alpha}}\, \td
%\\
&\leq \left(1+\td L_G\right)^n\left|E_d^{0}\right| + {\hat{\alpha}}\, \td \sum_{m=0}^{n-1} \left(1+\td L_G\right)^m
\\ 
%&\leq  {\hat{\alpha}}\, \td  \sum_{m=0}^{n-1} (1 + \td L_G)^m
%={\hat{\alpha}} \frac{(1 + \td L_G )^n - 1}{L_G}
%\\
\label{alt.alpha} &\leq \frac{{\hat{\alpha}}}{L_G}e^{n \td L_G}\; ,
\end{align}
where we used that at $n=0$ we initialize with $E^0_d=0$. Recalling the bound (\ref{e.alphan}) for $|\alpha_n|$ and employing Lemma~\ref{g-G} for $0<\dt<\frac{2\eps}{\lambda+1}$ we obtain for the bound of the discretization error
\begin{align} \label{final_eqn_E}
 |E_d^n|\leq \frac{e^{n \td L_G}}{L_G} 
 \Bigg\{  
2C^* \td
 + L_g \left(\frac{\varepsilon}{\td} + 3L_g(1+L_f)\eps + e^{-\frac{M\dt}{\varepsilon}} \right)|d^{n}| + 3 L_g L_f C_g \varepsilon  
\Bigg\} \; .
\end{align}
Employing Lemma~\ref{g-G} for $\frac{2\eps}{\lambda+1} < \dt < \frac{2\eps}{\lambda}$ we analogously obtain 
\begin{align} 
\nonumber |E_d^n|\leq &\frac{e^{n \td L_G}}{L_G} 
 \Bigg\{  
2C^* \td
 + L_g \left(\frac{\dt}{\dt^*}\frac{\varepsilon}{\td} + 3\frac{\dt}{\dt^*}L_g(1+L_f)\eps + e^{-\frac{M\dt^*}{\varepsilon}} \right)|d^{n}|
 \\
\label{final_eqn_E.aux}& \qquad \qquad + 3 \frac{\dt}{\dt^*}L_g L_f C_g \varepsilon
\Bigg\} \; .
\end{align}
Theorem~\ref{theorem.Ed} follows realizing $\dt^\star<\dt$ for $\dt>2\eps/(\lambda+1)$.
\end{proof}

\medskip
\noindent
%Note that $\frac{\eps}{\td} < 1$ as $\eps < \Dt$ from ({\it A6}). 
Besides the parameters used in the numerical scheme, i.e. the macrostep size $\Dt$, the number of microsteps $M$ with microstep size $\dt$, and the time scale parameter $\eps$, the error bound also involves the maximal deviation $|d^n|$ of the fast variables to the approximate centre manifold.

\noindent The bound on $|d^n|$ is established in the following Lemma.

%%%%%%%%%%%%%%%%%%%%%%%%%%%%%%%%
\medskip
\begin{lemma}
\label{lemma5} Assuming  ({\it A1}), ({\it A4}), ({\it A6}) and ({\it A8}), the maximal deviation of the fast variables from the approximate centre manifold over $n$ macrosteps $|d^n| = \max_{0\leq i\le n-1} |d^{n,0}|$ satisfies
\begin{align*}
\left|d^n\right| \leq 
\twoparts
{\left| d^{0,0} \right| + \frac{\eps L_f C_g (1 + \lambda) \Dt }{\eps-\Dt\lambda e^{-\frac{M\dt}{\eps}} }}
{\left|d^{0,0}\right| +  \frac{\eps L_f C_g \left(1+\lambda\frac{\dt}{\dt^*}\right) \Dt}{\eps- \Dt\lambda e^{-\frac{M\dt^*}{\eps}} }} \;\;.
\end{align*}
\end{lemma}

\begin{proof}
We bound for $0\le i\le n$
\begin{align*}
| d^{i,0}| &= |x^i - \bar{f}(y^i)|
 \\
 &= \left|x^{i-1,M} - \frac{ \Dt}{ \varepsilon}\left(\Lambda \, x^{i-1,M} - {f}(y^{i-1,M}) \right) - \bar{f}(y^i) \right| %\,\,\, \text{(from \eqref{PImacro_x})}
 \\
 %&= \left|\left(\I -\frac{\Dt}{\eps}\Lambda\right) x^{i-1,M} +\frac{\Dt}{\eps}\bar{f}(y^{i-1,M}) -\bar{f}(y^{i-1,M}) + \bar{f}(y^{i-1,M}) -\bar{f}(y^i) \right| 
 %\\
 %&= \left| \left(\I -\frac{\Dt}{\eps}\Lambda\right)d^{i-1,M}+ \bar{f}(y^{i-1,M}) -\bar{f}(y^i)\right|
 %\\
 &\leq \Dt \frac{\lambda}{\eps} \left|d^{i-1,M} \right| + \left| \bar{f}(y^{i-1,M}) -\bar{f}(y^i)\right| \; ,
\end{align*}
%The first term predicts that the distance between the fast variable and the manifold at time $t^{n-1,M}$, $\left|d^{n-1,M} \right|$, will decrease proportionally to $\frac{\Dt}{\eps}$. Since $\Dt$ is typically much larger than $\eps$, this is typically an expansion term and has been written to reflect that. The second term measures the change in the manifold over the macrostep (see Figure \ref{fig:R} for a diagram showing this relationship). \\
where we have used assumption ({\it A6}) to simplify
\[
\left|\Dt \frac{\Lambda}{\eps} - \I  \right| = \Dt \frac{\lambda}{\eps} - 1 < \Dt \frac{\lambda}{\eps} \; .
\] 
Employing Lemma~\ref{lemma_dnm} for $0<\dt<\frac{2\eps}{\lambda+1}$, and assumptions ({\it A1}) and ({\it A4}), we obtain
\begin{align}
\label{dn.alt} | d^{i,0}| 
&\leq \Dt \frac{\lambda}{\eps} \left|d^{i-1,M} \right| + L_f\left| y^{i-1,M} -y^i \right| 
\\
%&\leq \Dt \frac{\lambda}{\eps} \left( \left(1-\frac{\dt}{\eps} \right)^M \left|d^{i-1,0} \right| + L_f C_g \eps \right) + L_f C_g \Dt
%\nonumber 
%\\
%\label{dn0_macro} 
\nonumber
&\leq \Dt \frac{\lambda}{\eps}e^{-\frac{M\dt}{\eps}} \left|d^{i-1,0} \right| + L_f C_g (1+\lambda) \Dt \; .
 \end{align}
 Evaluating this recursive relationship yields
 \begin{align}
 | d^{i,0}|  
 \label{eq_dt_rec}
 &\leq \left[ \Dt \frac{\lambda}{\eps}e^{-\frac{M\dt}{\eps}} \right]^i \left|d^{0,0}\right| + L_f C_g (1+\lambda) \Dt \frac{1-\left[ \Dt \frac{\lambda}{\eps}e^{-\frac{M\dt}{\eps}} \right]^i}{1- \Dt \frac{\lambda}{\eps}e^{-\frac{M\dt}{\eps}} }
 \\
 %&\leq \left[ \Dt \frac{\lambda}{\eps}e^{-\frac{M\dt}{\eps}} \right]^i \left|d^{0,0}\right| +  \frac{L_f C_g  (1+\lambda) \Dt}{1- \Dt \frac{\lambda}{\eps}e^{-\frac{M\dt}{\eps}} } \\
 \nonumber
  &\leq \left|d^{0,0}\right| +  \frac{L_f C_g (1 + \lambda) \Dt}{1- \Dt \frac{\lambda}{\eps}e^{-\frac{M\dt}{\eps}} } \; ,
\end{align}
where we used assumption ({\it A8}) with $0<\dt<\frac{2\eps}{\lambda+1}$. Using Lemma~\ref{lemma_dnm} for $\frac{2\eps}{\lambda+1} < \dt < \frac{2\eps}{\lambda}$ in \eqref{dn.alt} we obtain
\begin{align}
 | d^{i,0}| 
 \nonumber
 &\leq \Dt \frac{\lambda}{\eps}e^{-\frac{M\dt^*}{\eps}} \left|d^{i-1,0} \right| + L_f C_g \left(1+\lambda\frac{\dt}{\dt^*}\right) \Dt 
 \\
 \label{eq_dtstar_rec}
 &\leq \left[ \Dt \frac{\lambda}{\eps}e^{-\frac{M\dt^\star}{\eps}} \right]^i \left|d^{0,0}\right| + L_f C_g (1+\lambda\frac{\dt}{\dt^\star}) \Dt \frac{1-\left[ \Dt \frac{\lambda}{\eps}e^{-\frac{M\dt^\star}{\eps}} \right]^i}{1- \Dt \frac{\lambda}{\eps}e^{-\frac{M\dt^\star}{\eps}} }
 \\
&
\nonumber
\leq \left|d^{0,0}\right| +  \frac{\eps L_f C_g \left(1+\lambda\frac{\dt}{\dt^*}\right) \Dt}{\eps- \Dt \lambda e^{-\frac{M\dt^*}{\eps}} } \;, 
\end{align}
where we used assumption ({\it A8}) with $\frac{2\eps}{\lambda+1} < \dt < \frac{2\eps}{\lambda}$.
\end{proof} 
\medskip
\begin{remark}
In the limits of Assumption ({\it A8}), $\Dt  \, \exp\left(-\frac{M\dt}{\eps}\right) \to  \frac{\eps}{\lambda}$ and $\Dt  \, \exp\left(-\frac{M\dt^\star}{\eps}\right) \to \frac{\eps}{\lambda}$, the terms $1-\left[ \Dt \frac{\lambda}{\eps}e^{-\frac{M\dt}{\eps}} \right]^i$ in \eqref{eq_dt_rec} and $1-\left[ \Dt \frac{\lambda}{\eps}e^{-\frac{M\dt^\star}{\eps}} \right]^i$ in \eqref{eq_dtstar_rec}, which are neglected in Lemma~\ref{lemma5}, are crucial to obtain finite estimates $|d^{n}|\le \left|d^{0,0}\right| +  n\Dt \,L_f C_g \left(1+\lambda \right)$ and $|d^{n}|\le \left|d^{0,0}\right| +  n\Dt \,L_f C_g \left(1+\frac{\dt}{\dt^*}\lambda\right)$, respectively.
\end{remark}

\medskip 
Theorem~\ref{theorem.main} follows by combining Theorems~\ref{theorem.Ec}~and~\ref{theorem.Ed} with Lemma~\ref{lemma5}, realizing $\dt^\star<\dt$ for $\dt>2\eps/(\lambda+1)$.

%%%%%%%%%%%%%%%%%%%%%%%%%%%%%%%%%%%

\section{Numerical confirmation of the error bound $|E_d^n|$}
\label{sec:numerics}

We now illustrate the error bound $|E_d^n|$ of Theorem~\ref{theorem.Ed}, \eqref{final_eqn_E}, which we recall here including the constants as obtained in the proof
%\begin{align}
%|E_d^n|\leq \frac{e^{n \td L_G}}{L_G} 
% \Bigg\{  
% 2C^* t_\Delta 
% + {L_g}\left(\frac{\varepsilon}{\td} + 3 L_g (1+L_f)\eps + e^{-\frac{M\dt}{\varepsilon}} \right)|d^{n}|
%+ 3 L_g L_f C_g \varepsilon  
%\Bigg\} \; ,
%\label{e.Edn}
%\end{align}
\begin{align}
|E_d^n|\leq \,3\frac{e^{n \td L_G}}{L_G} 
 \Bigg\{  
 C^* t_\Delta 
 + {L_g}\left(\frac{\varepsilon}{\td} + L_g (1+L_f)\eps + e^{-\frac{M\dt}{\varepsilon}} \right)|d^{n}|
+ L_g L_f C_g \varepsilon  
\Bigg\} \; ,
\label{e.Edn}
\end{align}
for $0<\dt<\frac{2\eps}{\lambda+1}$, and 
%\begin{align} 
%\nonumber |E_d^n|\leq &\frac{e^{n \td L_G}}{L_G} \frac{\dt}{\dt^*}
% \Bigg\{  
%2C^* \td
% + L_g \left(\frac{\varepsilon}{\td} + 3L_g(1+L_f)\eps + e^{-\frac{M\dt^*}{\varepsilon}} \right)|d^{n,0}|
% + 3 L_g L_f C_g
%\Bigg\} \; ,\\\label{e.Edn.aux}
%\end{align}
\begin{align} 
\nonumber |E_d^n|\leq \,3\frac{e^{n \td L_G}}{L_G} \frac{\dt}{\dt^*}
 \Bigg\{  
C^* \td
 + L_g \left(\frac{\varepsilon}{\td} + L_g(1+L_f)\eps + e^{-\frac{M\dt^*}{\varepsilon}} \right)|d^{n}|
 + L_g L_f C_g\eps
 %\\
%\label{e.Edn.aux}%& + 3 \frac{\dt}{\dt^*}L_g L_f C_g \varepsilon
\Bigg\} \; ,\\\label{e.Edn.aux}
\end{align}
for $\frac{2\eps}{\lambda+1}<\dt<\frac{2\eps}{\lambda}$.\\
We demonstrate the scaling of $|E^n_d|$ with respect to the macrostep size $\Dt$, the timescale parameter $\eps$ and the reinitialization error of the fast variable $|d^n|$ for fixed final time $T=n\td$. Note that the dependencies of $\eps$ and $\Dt$ are complicated through the dependency of $|d^n|$ on those parameters. We show results for simulations using the following multiscale system
\begin{align}
\label{toy_y}\dot{y_\eps} &= -x_\eps y_\eps -ay_\eps^2 \\
\label{toy_x}\dot{x_\eps} &= \frac{-x_\eps+\sin^2(by_\eps)}{\eps} \; ,
\end{align}
which has, at lowest order in $\eps$, the slow limit system
\begin{align}
\label{toy_Y} \dot{Y} &= -Y\sin^2(bY) - aY^2 \; .
\end{align}
For higher order approximations of the centre manifold and the associated coordinate transformations relating $y$ and $Y$ the reader is referred to the very useful webpage \cite{RobertsWeb} (see also \cite{Roberts08}).\\ The system (\ref{toy_y})-(\ref{toy_x}) with initial condition $y_\eps(0) > 0$ is locally Lipschitz with Lipschitz constant $L_f \leq b$ and $L_g={\rm{max}}(|x_\eps|+2a|y_\eps|)$ where the maximum is taken over the local region around the fixed point at $(x_\eps,y_\eps)=(0,0)$ under consideration. The vector field of the slow dynamics (\ref{toy_y}) is locally bounded by $C_g = \max( |x_\eps y_\eps|+a|y_\eps|^2)$, with the maximum taken over the same region. The free parameters $a$ and $b$ are used to control the Lipschitz constants $L_f$ and $L_g$. Here $\lambda=1$, implying $\frac{2\eps}{\lambda+1} = \eps$. Therefore \eqref{e.Edn} holds for $0<\dt<\eps$ and \eqref{e.Edn.aux} holds for $\eps<\dt<2\eps$.\\

\noindent
We first investigate how $|E_d^n|$ scales with the macrostep $\Dt$. We ensure that the term proportional to $C^* \td$ is the dominant term in (\ref{e.Edn}) by initializing the fast variables on the approximate slow manifold with $|d^{0,0}| = 0$. 
%To further control the distance to the centre manifold $|d^n|$  we choose parameters $\Dt/\eps \exp(-M \dt/\eps)<1$ (see Lemma~\ref{lemma5}). 
%The time of integration $T=n\td$ is kept fixed for all values of $\Dt$ by adjusting $n$, the number of iterations of PI. 
Figure \ref{Dt_scaling} illustrates the linear scaling of $|E_d^n|$ with $\Dt$. System parameters are $a=1$, $b=0.1$. The scale separation parameter is $\eps=10^{-5}$. We used $M=90$ microsteps with microstep size $\dt = 0.1\eps$ and $\dt=1.6\eps$. The number of iterations $n$ varied from $48$ to $918$ to keep $T=n\td=1$ fixed for all values of $\Dt$. Initial conditions are chosen to lie on the approximate slow centre manifold with $y^0 = 1$, $x^0 = \sin^2(0.1)$. The Lipschitz constants are $L_g=2$ and $L_f = 0.2$, the bound on the vector field of the slow dynamics is $C_g=1$, and the maximal second derivative of the reduced slow dynamics is $C^*=2$.\\

\begin{figure}
\includegraphics[width=\textwidth]{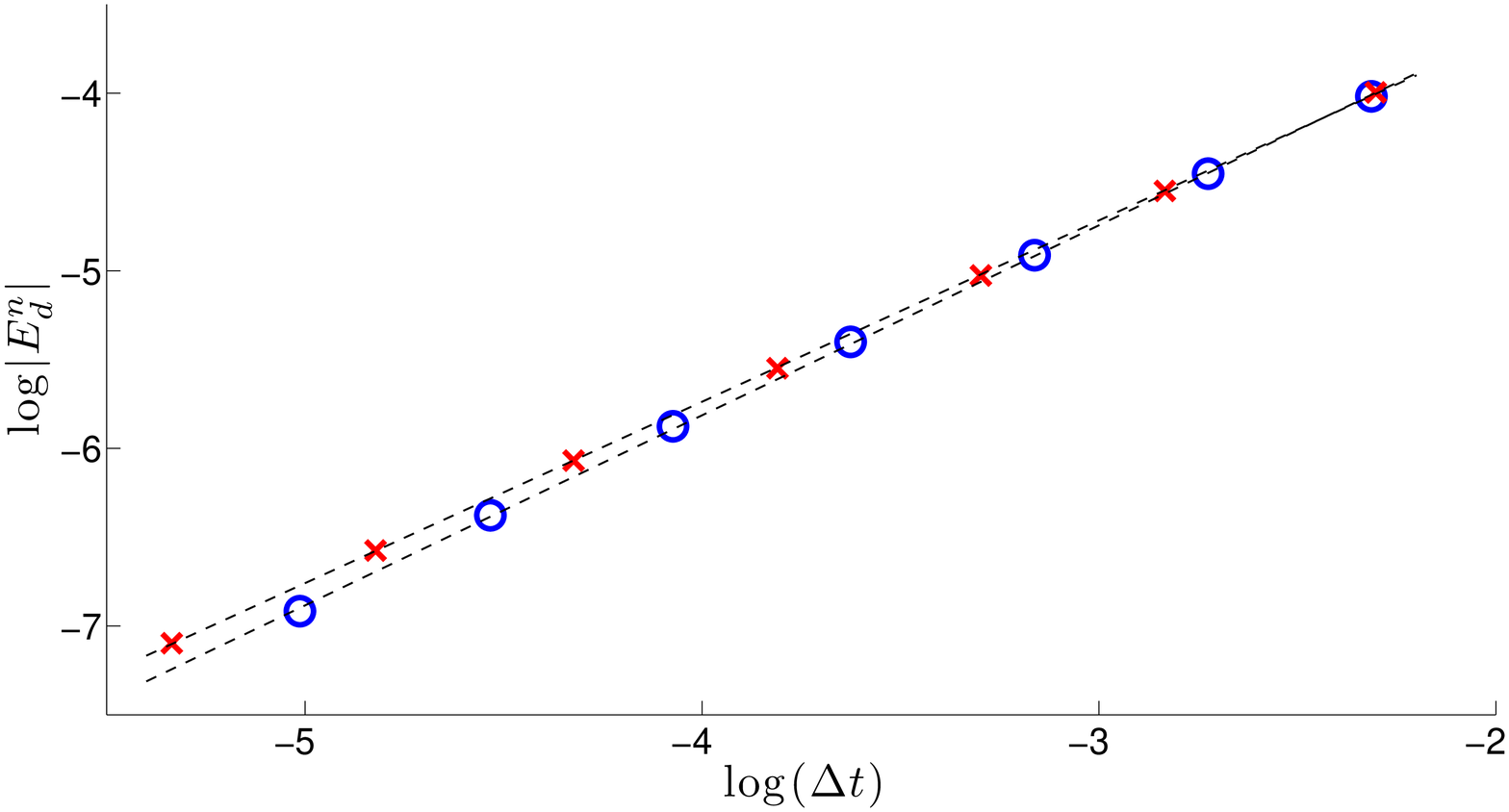}
\caption{Plot of $\log |E_d^n|$ versus $\log \Dt$ for fixed time of integration $T=1$. The dots represent results from numerical PI simulations of the system \eqref{toy_y}-\eqref{toy_x}, with the crosses representing results from a system with $\dt=0.1\eps$ and the circles representing a system with $\dt=1.6\eps$. The dashed lines are linear regression lines with a slopes of $1.02$ and $1.07$, respectively.} 
\label{Dt_scaling}
\end{figure}

\noindent
We now present results for the scaling of $|E_d^n|$ with the time scale parameter $\eps$. To focus on the linear scaling suggested by the term $L_g L_f C_g \eps$ in (\ref{e.Edn}), we will have to control the term proportional to $|d^n|$ and the $C^*\td$ term. The $|d^n|$-term is controlled by setting initial conditions on the centre manifold and employing $M\gg 1$, allowing for relaxation to the centre manifold. The $C^*\td$ term we control by choosing $\Dt<\eps$, violating assumption ({A6}) (note that this implies $\dt<\eps$). System parameters are $a=0.1$, $b=1$. Initial conditions are $y^0 = 5$, $x^0 = \sin^2(5)$, i.e. $|d^{0,0}|=0$. We used $M=100$ microsteps with microstep size $\dt = 10^{-6}$, macrostep size $\Dt=10^{-4}$, for $n=50$ iterations of PI with $T=0.01$. The Lipschitz constant for the slow dynamics is $L_g=2$ and for the centre manifold is $L_f = 1$, the bound on the vector field of the slow dynamics is $C_g=7$, and the maximal second derivative of the reduced slow dynamics is $C^*=6$. Figure \ref{new_eps} illustrates the linear dependence of $|E_d^n|$ on $\eps$ in this situation.\\

\begin{figure}[h!]
\includegraphics[width=\textwidth]{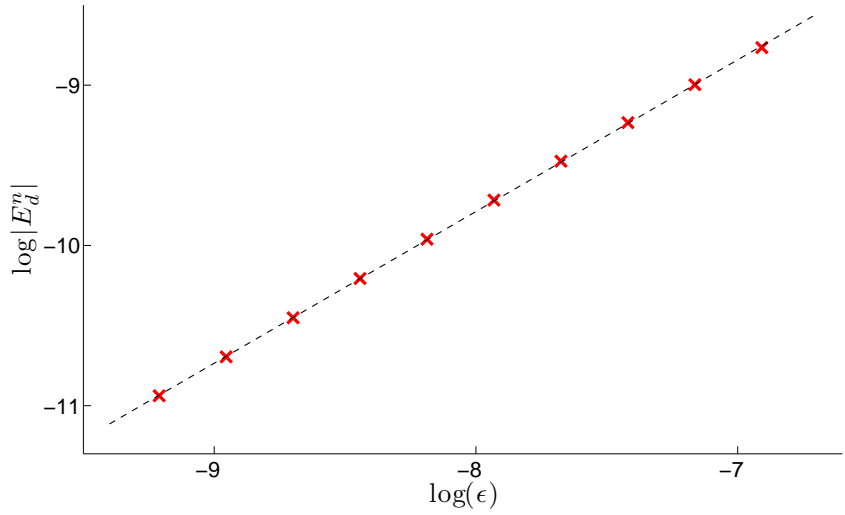}
\caption{
Plot of $\log |E_d^n|$ versus $\log \eps$. The dots represent results from numerical PI simulations of the system \eqref{toy_y}-\eqref{toy_x}. The line is a linear regression with a slope of $0.947$. 
}
\label{new_eps}
\end{figure}

\noindent
We now illustrate the linear scaling of $|E_d^n|$ with the maximal distance $|d^n|$ of the fast variable from the approximate centre manifold after a macrostep. To ensure that the error is not dominated by the initial initialization error $|d^{0,0}|$, we choose parameters which render the scheme unstable and violate assumption ({\it A8}), allowing for divergence of the fast variables from the centre manifold over the macrosteps, i.e. $|d^{n,0}| > |d^{0,0}|$. Figure \ref{dn_scaling} shows clearly the linear dependence of $|E_d^n|$ on $|d^n|$. System parameters are $a=1$, $b=1$. Initial conditions are $y^0 = 1$, $x^0 \in [\sin^2(1)+0.01, \sin^2(1)+0.5]$. The scale separation is $\eps=10^{-4}$. We used $M=100$ microsteps with microstep size $\dt = 0.01\eps$ and $\dt=1.99\eps$, and $n=5$ macrosteps with macrostep size $\Dt=10^{-3}$ implying $T=0.0055$ or $T=0.1$. The Lipschitz constants are $L_g=1$ and $L_f = 1$, the bound on the vector field of the slow dynamics is $C_g=290$, and the maximal second derivative of the reduced slow dynamics is $C^*=6$. 
%{\color{blue} We additionally comment that the error scalings for $\dt<\eps$ and $\dt>\eps$ are very similar despite the much larger final time $T$ of the simulation with $\dt>\eps$.} 
\\ 

\begin{figure}[h!]
\includegraphics[width=\textwidth]{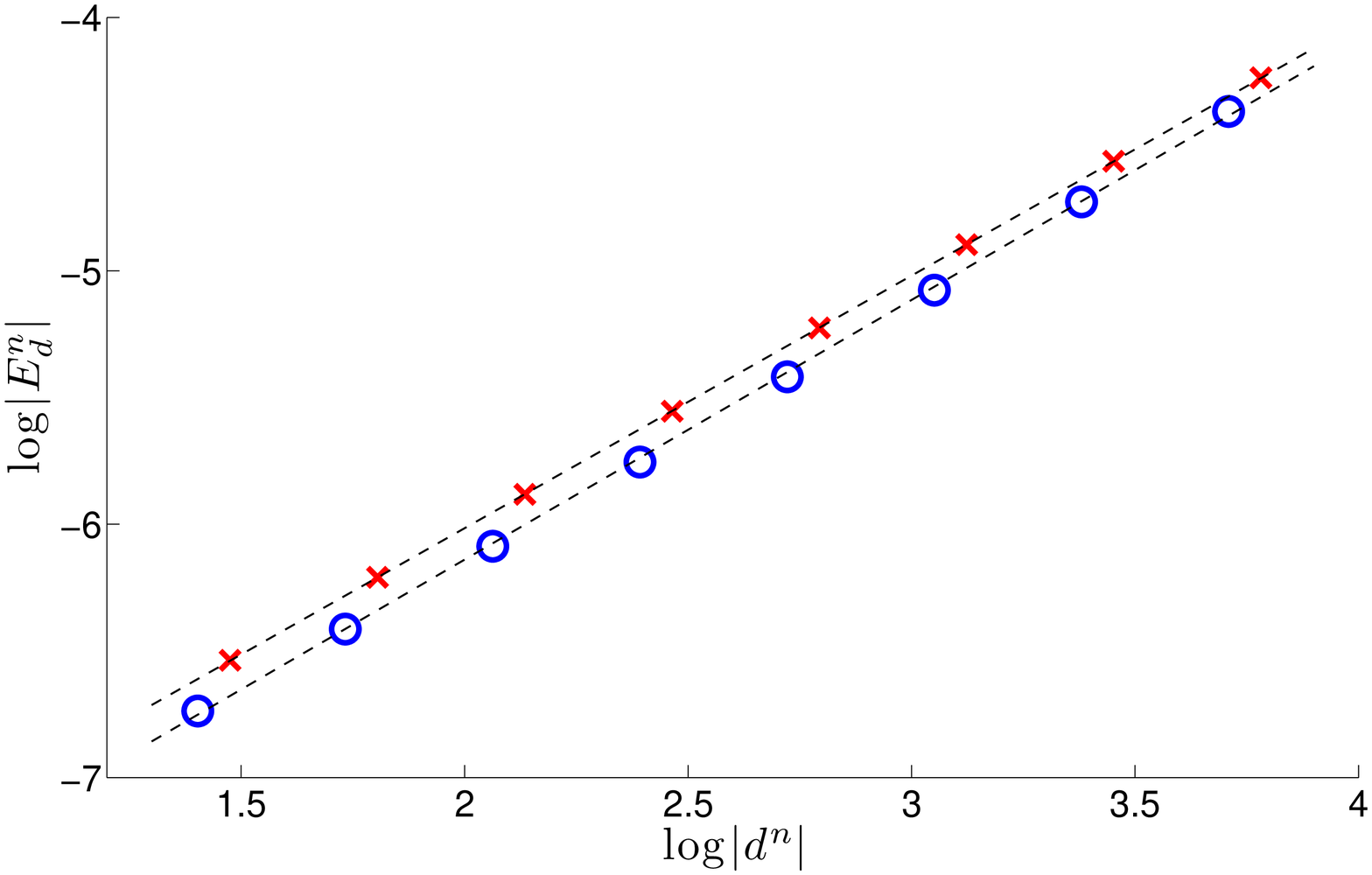}
\caption{
Plot of $\log |E_d^n|$ versus $\log |d^n|$. The dots represent results from numerical PI simulations of the system \eqref{toy_y}-\eqref{toy_x}. The crosses are results from simulations with $\dt=0.01\eps$ and the circles are results from simulations with $\dt=1.99\eps$. The dashed lines are linear regression lines with a slope of $1.00$ and $1.03$, respectively.}
\label{dn_scaling}
\end{figure}

%%%%%%%%%%%%%%%%%%%%%%%%%%%%%%%%%%%

\section{Discussion}
\label{sec:summary}
We have established bounds on the error of a numerical approximation of the solution of a multiscale system by Projective Integration. The error contains terms stemming from the inherent error made by reducing the full dynamics on the centre manifold as well as errors specific to the numerical discretization. In particular, the order of the numerical scheme features, as well as errors due to inaccurately approximating the dynamics on the centre manifold. Although the constants involved in our error estimates are not optimal, the numerical simulations suggest that the scaling obtained is correct.\\ 
%The range of micro time steps $\dt<2 \eps/(\lambda+1)$ which was suggested by our proof, illustrates a computational advantage in stiff systems with $\lambda\gg1$: Theorem~\ref{theorem.main} and Lemma~\ref{lemma_dnm} assure that one does not need to resolve the fastest time scale employing a computationally more stringent $\dt<\eps/\lambda$.} 
%{\color{red}{\note{Are we sure about the constants not being bigger then? Do you have numerical confirmation for that? Can you produce that quickly, please. Just take a $3d$-system.}}}
In future work it is planned to use the analytical results obtained here as well as a trivial extension of our results to seamless HMM and the results obtained for the non-seamless version of HMM (see \cite{E03}) to shed light on the important question in what circumstances one method or the other may lead to better performance. These methods exhibit different error bounds due to their different weights as well as due to differing reinitialization procedures for the fast variables.

%%%%%%%%%%%%%%%%%%%%%%%%%%%%%%%%%%%

\section*{Acknowledgments}

We are grateful for discussions with Daniel Daners and Tony Roberts. GAG acknowledges support from the Australian Research Council. John Maclean is supported by a University of Sydney Postgraduate Award.

%%%%%%%%%%%%%%%%%%%%%%%%%%%%%%%%%%%
\bibliographystyle{natbib}
\bibliography{bibliography}
%%%%%%%%%%%%%%%%%%%%%%%%%%%%%%%%%%%
\end{document}